\documentclass[11pt]{amsart}
\usepackage{amsmath,amssymb,amsthm}
\usepackage{color,graphics,srcltx}
\usepackage{enumitem}

\newtheorem{theorem}{Theorem}[section]
\newtheorem{lemma}[theorem]{Lemma}
\newtheorem{corollary}[theorem]{Corollary}

\newtheorem{remark}[theorem]{Remark}

\theoremstyle{definition}
\newtheorem{de}[theorem]{Definition}
\newtheorem{example}[theorem]{Example}

\newcommand\C{\mathcal C}
\newcommand\N{\mathcal N}

\newcommand\ZZ{\mathbb Z}
\DeclareMathOperator{\PSL}{{\mathrm{PSL}}}
\DeclareMathOperator{\SL}{{\mathrm{SL}}}
\DeclareMathOperator{\GL}{{\mathrm{GL}}}
\DeclareMathOperator{\GF}{{\mathrm{GF}}}
\DeclareMathOperator{\Aut}{{\mathrm{Aut}}}
\newcommand\norml{\trianglelefteq}

\usepackage{color}

\begin{document}
\title{Realizability problem for commuting graphs}
\author{Michael Giudici}
\address[Michael Giudici]{Centre for the Mathematics of Symmetry and Computation, The University of Western Australia, 35 Stirling Highway, Crawley, WA 6009, Australia.}
\thanks{This paper is part of an Australian  Research Council Discovery Project (DP 120100446).}
\email{michael.giudici@uwa.edu.au}

\author{Bojan Kuzma}
\address[Bojan Kuzma]{${}^1$University of Primorska, Glagolja\v ska 8, SI-6000 Koper,
Slovenia, \hbox{and} ${}^2$IMFM, Jadranska 19, SI-1000 Ljubljana,
Slovenia}
\thanks{Work of the second author is supported in part by the Slovenian Research Agency (Research
cooperation Slovenia - Australia)}\email{bojan.kuzma@famnit.upr.si}

\begin{abstract}
We investigate properties which ensure that a given finite graph is the commuting graph of a  group or
semigroup.
We show that all graphs on at least two vertices such that no vertex is adjacent to all other
vertices is the commuting graph of some semigroup. Moreover, we obtain a complete classification of the
graphs with an isolated vertex or edge that are the commuting graph of a group and the cycles that are the
commuting graph of a centrefree semigroup.

\end{abstract}
\maketitle
\section{Introduction and Preliminaries}

Let $S$ be a semigroup with centre $Z(S):=\{x\in S\mid xs=sx\hbox{ for all } s\in S\}$. The \emph{commuting
graph} $\Gamma(S)$ is the simple graph with vertex set $S\backslash Z(S)$  and two distinct vertices $x$ and
$y$ are adjacent if and only  $xy=yx$. This notion can be traced back at least as far as the paper by Brauer
and Fowler \cite{brauer-fowler} who used commuting graphs to study the distances between involutions in
finite groups (it should be mentioned, however,  that the vertices of their graph consisted of all
nonidentity elements).

Solomon and Woldar~\cite{Solomn-Woldar} showed that a commuting graph distinguishes  finite simple
nonabelian groups. More precisely, if the commuting graph of  a group is isomorphic to the commuting
graph of some finite  nonabelian simple group then the two groups are isomorphic. In general, commuting
graphs do not distinguish groups  as the commuting graph of $Q_8$ and $D_8$ both consist of three
disjoint edges.

In the present paper we will be concerned with the following inverse problem for commuting graphs: Given a
simple graph $\Gamma$, can we find a group or semigroup whose commuting graph is isomorphic to $\Gamma$? We
say that such a graph $\Gamma$ is \emph{realizable} over groups or over semigroups, respectively. Part of
the motivation behind the present study are the results of Pisanski \cite{pisanski} who  showed that any
graph on $n$ vertices is isomorphic to the commuting graph of a certain subset of  $S_3^n$, a direct product
of $n$ copies of the symmetric group $S_3$. However, this subset is in general not a semigroup as it may
not be closed under multiplication. Unaware of \cite{pisanski}, the second author with
collaborators~\cite{ambrozie-bracic-kuzma-muller} proved, among other things, that every finite simple graph
is isomorphic to the commuting graph of a certain subset of sufficiently large complex matrices. It was
further shown in~\cite{ambrozie-bracic-kuzma-muller} that for any given size $n$ of matrices there exists a
finite graph which is not realizable as a commuting graph of a subset of $n$--by--$n$ matrices.

Ara\'ujo, Kinyon and Konieczny \cite{AKK} posed the problem of classifying the commuting graphs of
semigroups. We will  give a complete answer to this question in Theorem \ref{thm:semigroup} by showing that
any finite graph on at least two vertices which does not have a vertex  adjacent to all other vertices is
the commuting graph of a semigroup. This gives an alternative proof for the result of Ara\'ujo, Kinyon and
Konieczny \cite{AKK} that for every integer $n$ there is a semigroup whose commuting graph has diameter $n$.
The semigroup in their proof has no centre while the semigroup in our construction to prove Theorem
\ref{thm:semigroup} has a  centre of order two.  We undertake some investigations of graphs that are
realizable over centrefree semigroups with our main result being the following.

\begin{theorem}\label{thm:cycle-semigps}
A cycle is the commuting graph of a centrefree semigroup if and only if its length is divisible by four.
\end{theorem}

Commuting graphs of groups are much more restrictive.  For example, a GAP \cite{GAP4} computation
shows that the smallest group with a connected commuting graph has order 32. There are seven such
groups; each has a centre of order two, and hence these yield connected commuting graphs on 30
vertices. If $G$ were a group with connected commuting graph on less then 30 vertices then since $|Z(G)|$ divides $|G|$ we would have  $|G|<60$ but a GAP  calculation shows that
no such group exists.

There are also many other restrictions on the possible graphs realizable as the commuting graph of a group.
Indeed, even though there is no bound for the diameter of the commuting graph of a group
\cite{Giudici_Parker}, Morgan and Parker \cite{Morgan-Parker} have shown that every connected component of
the commuting graph of a group with trivial centre has diameter at most 10. Moreover, they showed that if
$\Delta$ is a connected component of such a graph for a nonsoluble group  with $\Delta$  containing no
involutions, then $\Delta$ must be a clique. Furthermore, Afkhami, Farrokhi and Khashyarmanesh \cite{AFK}
have shown that only 17 groups have a planar commuting graph. This result was also obtained by Das and
Nongsiang \cite{DasNong}, who further proved that for a given genus $g$ there are only finitely many groups
whose commuting graph has genus $g$. In addition,  \cite{DasNong}   also shows that only three groups have
triangle free commuting graphs.

Section \ref{sec:groups} collects some simple observations about the structure of commuting graphs of
groups. In particular, Lemma \ref{lem:edgeintriangle} shows that any non-isolated edge of the commuting
graph of a group must lie in a triangle. We then go on to determine the structure of commuting graphs of
groups that contain an isolated vertex or edge. This work is summarised in the following two theorems.

\begin{theorem}
\label{thm1}
Let $G$ be a group and suppose that $\Gamma(G)$ has an isolated vertex. Then $\Gamma(G)$ has exactly $|G|/2$ isolated vertices and the remaining vertices form a clique.
\end{theorem}

\begin{theorem}\label{thm2}
Let $G$ be a group and suppose that $\Gamma(G)$ has an isolated edge. Then $\Gamma(G)$ consists of isolated edges, cliques and at most one noncomplete connected component. Moreover, such a component has diameter at most 5.
\end{theorem}

More specific information about the groups involved in Theorem \ref{thm1} is given in Theorem \ref{thm:isolated-vertex}, while more details about the graphs and groups involved in Theorem \ref{thm2} are given in Lemma \ref{lem:Gtrivialcentre}, Theorem \ref{thm:14}, Lemma \ref{lem:edgenontrivcent} and Theorem \ref{thm:edge-center}.

\subsubsection{Notation} Given a graph $\Gamma$ we denote the vertex set of $\Gamma$ by $V(\Gamma)$ and the edge set by $E(\Gamma)$. Also, we denote by $|V(\Gamma)|$ the cardinality of the
vertex set of $\Gamma$. Given vertices $x,y\in\Gamma$ we denote by $x\sim y$ the fact that they form
an edge in $\Gamma$.
 The distance,
$d(x,y)$ between connected vertices $x$ and $y$ is the length of a minimal path between $x$ and $y$. We set
$d(x,y)=\infty$ if there is no path from $x$ to $y$. We let $K_n$ denote the complete graph on $n$ vertices.

Unless otherwise stated, all semigroups and groups are written  multiplicatively and the identity in a
group is $1$. Let $|G|$ denote the order of the (semi)group $G$ and, given $g\in G$, let $|g|$ denote
the order of $g$. Let $Z(G)$ be the centre of a (semi)group $G$, let $\C_G(A)$ be the centraliser of the
subset $A$ of a (semi)group $G$ and in the case where $G$ is a group let ${\mathcal N}_G(A):=\{x\in G\mid x^{-1}Ax=A\}$ be its normalizer.
Given elements $x,y\in G$, we denote by $\langle x,y\rangle$ the subgroup generated by $x$ and $y$, so
$\langle x\rangle$ is  the subgroup generated by $x$. By $H\leqslant G$  we denote that $H$ is a
subgroup of a group $G$, and  $\Aut (G)$ denotes the automorphism group of $G$.

Let $S_n$ and $A_n$ be the symmetric group and alternating group on $n$ elements, respectively, let
$\ZZ_p=\ZZ/(p\ZZ)$ be a cyclic group of order $p$, let $\SL(n,p^k)$ be the special linear group of $n\times
n$ matrices with determinant one over the Galois field $\GF(p^k)$, and let
$\PSL(n,p^k)=\SL(n,p^k)/Z(\SL(n,p^k))$ be the projective special linear group.  As usual, given a group
homomorphism $\phi\colon G\to G$, its action on an element $g\in G$ is denoted by $g\phi\in G$.

\section{Realizability over semigroups}

We start with two  basic obstructions that prevent realizability over semigroups.

\begin{lemma}\label{lem:obstruction}
Let $\Gamma$ be a graph. If either
\begin{enumerate}[label={\emph{(\roman*)}}]
 \item $|V(\Gamma)|=1$, or
 \item $\Gamma$ contains a vertex adjacent to all vertices in $V(\Gamma)\backslash\{v\}$,
\end{enumerate}
then $\Gamma$ is not the commuting graph of a semigroup.
\end{lemma}
\begin{proof}
Suppose that $\Gamma=\Gamma(S)$ for some semigroup $S$ and that $v\in V(\Gamma)$ such that either
$V(\Gamma)=\{v\}$ or $v$ is adjacent to all vertices in $V(\Gamma)\backslash\{v\}$. Then $v$ commutes with
itself and all elements of $Z(S)\cup \Gamma=S$ so $v\in Z(S)$,  a contradiction.
\end{proof}

In particular, Lemma ~\ref{lem:obstruction} shows that complete graphs are not commuting graphs of
semigroups. However, the obstructions in Lemma ~\ref{lem:obstruction} are the only ones that prevent
realizability over semigroups.

\begin{theorem}\label{thm:semigroup}
Every  finite graph $\Gamma$ with at least two vertices and such that no vertex is adjacent to all other
vertices is the commuting graph of some semigroup $S$ with $|S|=|\Gamma|+2$.
\end{theorem}
\begin{proof}
Let  $V(\Gamma)=\{v_1,\dots,v_n\}$ be the ordered vertex set of $\Gamma$.
Pick a set $Z=\{0,z\}$, with two distinct elements disjoint from  $V(\Gamma)$. On the set $S=V(\Gamma)\cup Z$ define a multiplication by $SZ=ZS=\{0\}$, and
$$v_iv_j=\begin{cases}
0; &\hbox{$i<j$ or $(v_i,v_j)$ is an edge in $\Gamma$}\cr
z;& \hbox{$i>j$ and $(v_i,v_j)$ is not an edge in $\Gamma$}.\cr
\end{cases}$$ It is easy to see (say, by associating to each $x\in S$ the transformation $(x\rho)\colon \{1\}\cup S\to \{1\}\cup S$,
defined by $(1)(x\rho)=x$ and $(s)(x\rho)=sx$, for each $s\in S$, and checking that multiplication defined
on $S$ is preserved) that (i) this multiplication is  associative, so $S$ is a semigroup with zero, (ii)
since no vertex in $\Gamma$ is adjacent to all other vertices, $Z(S)=Z$, and (iii)  $v_i,v_j\in S\setminus
Z(S)$ commute if and only if they form an edge in $\Gamma$.\footnote{Our original proof of Theorem
\ref{thm:semigroup} involved a semigroup with centre of order $n^2+1$ and the improvement to one with centre
of order 2 was kindly pointed out to us by Marcel Jackson.}
\end{proof}
 \begin{remark}
Since each set can be well ordered,     only  cosmetic   modification is required to show that every
infinite graph such that no vertex connects to all other vertices, is a commuting graph of a
semigroup.
\end{remark}

\begin{remark}
The upper bound on the order of $S$ is exact. Namely, it can be shown that the six-cycle graph $\Gamma=C_6$
is not the commuting graph of a semigroup of order seven.
\end{remark}

Things are more complicated if we restrict ourselves to centrefree semigroups, that is, to semigroups whose
centre is the empty set.  We demonstrate this with the next lemma and its corollary.

\begin{lemma}
\label{lem:labelling} Let $\Gamma$ be the commuting graph of a centrefree semigroup $S$ and let
$a,b,c\in V(\Gamma)$ such that
\begin{enumerate}[label={\emph{(\roman*)}}]
 \item $a\sim b\sim c$ and $a\not\sim c$,
 \item  no triangle in $\Gamma$ contains $\{a,b\}$ or $\{b,c\}$.
\end{enumerate}
Then
\begin{equation}\label{eq:cycle}
\hbox{either }ab=a\hbox{ and } bc=c\quad \hbox{ or }\quad ab=b=bc.
\end{equation}
\end{lemma}
\begin{proof}
Note first that $ab=ba\in S$ commutes with both  $a$ and $b$ and hence in $\Gamma=\Gamma(S)$, $ab$ is
adjacent to $a$ and $b$. By the assumption (ii), it follows that $ab\in\{a,b\}$ and similarly,
$bc=cb\in\{b,c\}$. Now, assume $ab=a$ and $bc=b$. Then, $ac=(ab)c=a(bc)=ab=a=ba=(cb)a=c(ba)=ca$,
contradicting the fact that $a\not\sim c$. Similarly, $ab=b$ and $bc=c$ is not possible since otherwise
$ac=a(bc)=(ab)c=bc=c=cb=c(ba)=(cb)a=ca$.
\end{proof}
\begin{corollary}\label{cor:oddcycle}
Let $\Gamma$ be the commuting graph of a centrefree semigroup $S$. Assume $\Gamma$ contains a cycle $C$ of
length $n\geq 4$ as an induced subgraph such that no two adjacent vertices of $C$ are contained in a
triangle from $\Gamma$. Then $n$ is even.
\end{corollary}
\begin{proof}
Suppose that $\Gamma$ contains a cycle $C$. By Lemma \ref{lem:labelling} the vertices of $C$ can be
labeled by $\{\pm1\}$ such that  if $ab=a$ and $bc=c$ we label vertex $b$ with $-1$, otherwise, if
$ab=b=bc$ we label vertex $b$ with $+1$. Hence walking around the cycle, the labeling alternates and
so $|C|$ is even.
\end{proof}

\begin{example}
The assumption in Corollary \ref{cor:oddcycle} that no edge from an induced cycle forms a triangle in
$\Gamma$ is essential. For example, the multiplication table
$$\hbox{\tiny$\bordermatrix{&s_1&s_2&s_3&s_4&s_5 &s_6\cr
    s_1 &s_1 &s_1 &s_1 &s_1 &s_1 &s_1 \cr
    s_2 &s_1 &s_1 &s_1 &s_1 &s_2 &s_2 \cr
    s_3 &s_1 &s_1 &s_1 &s_1 &s_3 &s_3 \cr
    s_4 &s_4 &s_4 &s_4 &s_4 &s_4 &s_4 \cr
    s_5 &s_1 &s_2 &s_2 &s_4 &s_5 &s_5 \cr
    s_6 &s_1 &s_3 &s_3 &s_4 &s_6 &s_6    \cr
                }$} $$
makes the set  $S=\{s_1,s_2,s_3,s_4,s_5,s_6\}$ into  a centrefree semigroup with $\Gamma(S)$ containing the
induced 5-cycle $s_2\sim s_3 \sim s_6 \sim s_4 \sim s_5\sim s_2$. This is possible because $s_1$ forms a
triangle with edges $\{s_5,s_2\}$, $\{s_2,s_3\}$ and $\{s_3,s_6\}$.
\end{example}
\begin{corollary}
If $n\ge 5$ then there exists a graph on $n$ vertices that is the commuting graph of a semigroup but
not the commuting graph of a centrefree semigroup.
\end{corollary}
\begin{proof}
If $n$ is odd consider an $n$-cycle. If $n$ is even, consider the disjoint union of an $(n-1)$-cycle and an isolated
vertex.
\end{proof}

\begin{example}
The edge-less graph $\Gamma$ on $n$ vertices is the commuting graph of the centreless semigroup
$S=\{v_1,\ldots,v_n\}$ with multiplication defined by $v_iv_j=v_i$. Note that associativity holds as
$v_i(v_jv_k)=v_iv_j=v_i$ and $(v_iv_j)v_k=v_iv_k=v_i$. Also $\C_S(v_i)=\{v_i\}$.
\end{example}
\begin{example}
GAP calculations show that each graph on at most four vertices that does not satisfy the obstructions in
Lemma \ref{lem:obstruction} is the commuting graph of some centrefree semigroup.
\end{example}

\begin{example}
GAP  calculations show that there exist exactly two simple graphs on $5$ vertices, each vertex with valency
at most three, which are not the commuting graph of a centrefree semigroup. One such graph is the 5-cycle
while the other is a \emph{house} 
with adjacency matrix $\left(
\begin{smallmatrix}
 0 & 1 & 1 & 1 & 0 \\
 1 & 0 & 1 & 0 & 1 \\
 1 & 1 & 0 & 0 & 0 \\
 1 & 0 & 0 & 0 & 1 \\
 0 & 1 & 0 & 1 & 0
\end{smallmatrix}
\right)$.

Note that, by Theorem~\ref{thm:semigroup}, both graphs are commuting graphs of a semigroup with nontrivial
centre.
\end{example}

Next we give a complete picture of when a cycle is the commuting graph of some centrefree semigroup and so proving Theorem \ref{thm:cycle-semigps}. It
turns out that if such a semigroup exists, it is essentially unique. Clearly, this is no longer the case if
the restriction about being centrefree is removed because $S$ and its unitization $S^1=S\cup\{1\}$ have the
same commuting graph. Note the following consequence of the Lemma below: there is no upper bound on the
diameter of a connected commuting graph of centrefree semigroups (see also \cite[Theorem 4.1]{AKK}).

Semigroups $S_1$ and $S_2$ are (anti)isomorphic if there exists a bijection $\phi\colon S_1\to S_2$ such
that $(ab)\phi=(a\phi)(b\phi)$ for each $a,b\in S_1$ (that is, $\phi$ is an isomorphism) or
$(ab)\phi=(b\phi)(a\phi)$ for each $a,b\in S_1$ (that is, $\phi$ is an antiisomorphism).

\begin{theorem}\label{lem:cycle-semigps}
A cycle is the commuting graph of some centrefree semigroup if and only if its length is divisible by four.
Up to (anti)isomorphism there exists at most one centrefree semigroup whose commuting graph is a given
cycle.
\end{theorem}

\begin{proof}
\newcounter{step}
\setcounter{step}{0} By Lemma~\ref{lem:obstruction}, a triangle is not the commuting graph of a
semigroup. So, by Corollary~\ref{cor:oddcycle} it only remains to consider even cycles.

Suppose therefore that an even cycle $C_{2k}$ is the commuting graph of a centrefree semigroup $S$. Each
vertex $x\in\Gamma(S)$ has exactly two neighbors, say $y,z$. It follows that
$x^2\in\C_S(\{y,x,z\})=\{x\}$, so $S$ is a band.

We label the vertices of $\Gamma(S)$ be $x_0,x_1,\ldots,x_{2k-1}$ such that $x_i\sim x_{i\pm1}$ where
addition of subscripts is done modulo $2k$. Now by Lemma~\ref{lem:labelling}, the identity
\eqref{eq:cycle} holds for $\Gamma(S)$ and so, as in the proof of Corollary~\ref{cor:oddcycle}, the
vertices of $\Gamma(S)$ can be labeled by $\{\pm1\}$. Without loss of generality we label the vertices
with even subscripts by $+1$ and those with odd subscripts by $-1$. Recall that this means that
\begin{equation}\label{eq:start-even-cycle-thm}
x_{2i}x_{2i\pm 1}=x_{2i\pm 1}x_{2i}=x_{2i} ;\qquad
i\in\{0,1,\dots,k-1\}.
\end{equation}

\addtocounter{step}{1} {\bf Step~\thestep}. If $x,y\in C_{2k}$ are at distance two then
\begin{equation}\label{eq:even-cycle}
xy\in\{x,y\}.
\end{equation}
To see this, let $x\sim b\sim y$ be a path of length two from $x$ to $y$. Note that $xy$ commutes with $b$ and so $xy\in\{x,b,y\}$. Suppose that $xy=b$. Then $xy$ commutes with $x$ and so
$$bx=(xy)x=x(xy)=xy=b.$$
Thus $b$ is labeled $+1$ and we also have $by=b$. Now $yx$ also commutes with $b$ and since $x\not\sim
y$ we must have $yx\in\{x,y\}$. If $yx=x$ then using the fact that $x^2=x$ we have that $xyx=x$ and so
$bx=x$, contradicting $b$ being labeled $+1$. Thus $yx=y$, but then $yxy=y$ and so $by=y$, another
contradiction. Thus (\ref{eq:start-even-cycle-thm}) holds.

\addtocounter{step}{1} {\bf Step~\thestep}. By \eqref{eq:even-cycle}, $x_0x_2\in\{x_0,x_2\}$ and by
considering, if necessary, the opposite semigroup (with multiplication given by $a\cdot b:=ba$) we may
assume that
\begin{equation}\label{eq:x0x2=x0}
x_0x_2=x_0.
\end{equation}

\addtocounter{step}{1} {\bf Step~\thestep}. We prove by induction that
\begin{equation}\label{eq:x2ix(2i+2)=x(2i)}
x_{2i}x_{2i+2}=x_{2i};\qquad i\in\{0,1,\dots,k-1\}.
\end{equation}
The base step is given by \eqref{eq:x0x2=x0}. To prove the induction step, assume we
have already shown that vertices $a,c$ with $d(a,c)=2$ and with labels $+1$ satisfy
\begin{equation}
ac=a.
\end{equation}
Choose $e\ne a$ with $d(c,e)=2$ and assume to the
contrary that
\begin{equation}\label{eq:ec=c}
ce=e.
\end{equation}
By~\eqref{eq:even-cycle}, $ca\in\{c,a\}$ and as $ca\ne ac=a$, we have
$$ca=c.$$ Likewise $ec=c$. Moreover, if $d\in\C_S(\{c,e\})$
is the unique vertex commuting with both $c$ and $e$ then,  since both $c,e$ have  label $+1$, so
$$dc=c=cd\quad\hbox{and}\quad ed=e=de,$$ and
from the inductive hypothesis \eqref{eq:ec=c} we deduce
\begin{equation}\label{eq:ad=a}
ad=(ac)d=a(cd)=ac=a.
\end{equation}
Furthermore, $d(da)=d^2a=da=d(ad)=(da)d$, so $da\in\C_S(d)=\{c,d,e\}$. Note that  $da=d$ is impossible
because then $da$ would commute with $c$ and we would get that $c=cd=c(da)=(da)c=d(ac)=da=d$, a
contradiction. Hence,
\begin{equation}\label{eq:ad<>d}
da\in\{c,e\}.
\end{equation} Now, by our assumption $cea=ea=eac$, and by~\eqref{eq:ad=a}
also $dea=ea=ead$ so $ea\in\C_S(\{c,d\})=\{c,d\}$. If $ea=d$ then, by~\eqref{eq:ad<>d},
$d=ea=ea^2=(ea)a=da\in\{c,e\}$, a contradiction. Hence,
$$ea=c.$$
Next,  in  \eqref{eq:ad<>d}, $da=e$ is impossible since then $e=ce=c(da)=(cd)a=ca=c$. So,
$$da=c.$$
For the unique $b\in\C(\{a,c\})$,  which clearly satisfies
$$ba=a=ab\quad\hbox{and}\quad bc=c=cb$$
this further gives $(bd)a=b(da)=bc=c$. Thus, $bd\ne b$ (because $ba=a=a^2$), and as
\eqref{eq:even-cycle} implies $bd\in\{b,d\}$ we see that
$$bd=d.$$
We also have $db\in\{b,d\}$ and since $d\not\sim b$ it follows that $db=b$.
However, we then have $a=ba=(db)a=d(ba)=d(ab)=(da)b=cb=c$, again a  contradiction. Hence, \eqref{eq:ec=c} is contradictory and we must have
    $ec=e$, which proves the induction step and hence the
    Eq.~(\ref{eq:x2ix(2i+2)=x(2i)}).

\addtocounter{step}{1} {\bf Step~\thestep}. We claim that
\begin{equation}\label{eq:ay=a}
x_{2i}y=x_{2i};\qquad \text{ for all }x_{2i},y\in S.
\end{equation}
In fact,  \eqref{eq:x2ix(2i+2)=x(2i)} implies $x_{0}x_{4}=(x_{0}x_2)x_4=x_0(x_2x_4)=x_0x_2=x_0$, and
by induction, $x_0x_{2i}=x_0$. Likewise we see that
$$x_{2i}x_{2j}=x_{2i};\qquad \text{ for all }0\le i,j\le k-1.$$
Furthermore, by \eqref{eq:start-even-cycle-thm} $x_{2j}x_{2j+1}=x_{2j+1}x_{2j}=x_{2j}$ which implies
$x_{2i}x_{2j+1}=(x_{2i}x_{2j})x_{2j+1}=x_{2i}(x_{2j}x_{2j+1})=x_{2i}x_{2j}=x_{2i}$, which proves
\eqref{eq:ay=a}.

\addtocounter{step}{1} {\bf Step~\thestep}. By \eqref{eq:ay=a},
$x_{2i+1}x_{2j}\in\C_S(x_{2i+1})=\{x_{2i},x_{2i+1},x_{2i+2}\}$ for all $i,j$. Actually, $x_{2i+1}x_{2j}=x_{2i+1}$ is
impossible since then, $x_{2i+1}x_{2i}=x_{2i}$ would give
$x_{2i}=x_{2i+1}x_{2i}=(x_{2i+1}x_{2j})x_{2i}=x_{2i+1}(x_{2j}x_{2i})=x_{2i+1}x_{2j}=x_{2i+1}$, a
contradiction. Hence
\begin{equation}\label{eq:x(2i+1)x(2j)}
x_{2i+1}x_{2j}\in\{x_{2i},x_{2i+2}\}.
\end{equation}
In particular, $x_{2i+1}x_{2j}$ has label $+1$ for all $i$ and $j$, and for a fixed $i$  can take only two
values as $j$ varies.

\addtocounter{step}{1}\newcounter{stepp}\setcounter{stepp}{\thestep} {\bf Step~\thestep}. Assume we
have $x_{2i+1}x_{2j}\ne x_{2i+1}x_{2j+2}$ for some $i,j$. We claim that then
\begin{equation}\label{eq:11}
x_{2i+1}x_{2j+1}=x_{2i+1}.
\end{equation} To
see this, we first show that the product $t:=x_{2i+1}x_{2j+1}$ cannot have label $+1$. Namely,
otherwise, by \eqref{eq:ay=a}, $ty=t$ for every $y\in S$. Combined with $x_{2j+1}x_{2j}=x_{2j}$ and
$x_{2j+1}x_{2j+2}=x_{2j+2}$ we would have
$$t=tx_{2j}=(x_{2i+1}x_{2j+1})x_{2j}=x_{2i+1}(x_{2j+1}x_{2j})=x_{2i+1}x_{2j},$$
and likewise
$$t=tx_{2j+2}=(x_{2i+1}x_{2j+1})x_{2j+2}=x_{2i+1}x_{2j+2},$$
which contradicts the fact that $x_{2i+1}x_{2j}\ne x_{2i+1}x_{2j+2}$. Thus $x_{2i+1}x_{2j+1}$ has
label $-1$ and as such commutes with exactly two vertices with label +1, namely $t_{\mathrm{prec}}$ and $t_{\mathrm{succ}}$. By \eqref{eq:ay=a} and \eqref{eq:x(2i+1)x(2j)}, given $s\in\{t_{\mathrm{prec}},
t_{\mathrm{succ}}\}$, we have
\begin{align*}
s&=s(x_{2i+1}x_{2j+1})=(x_{2i+1}x_{2j+1})s\in
x_{2i+1}\{x_{2j},x_{2j+2}\}\subseteq\{x_{2i},x_{2i+2}\}.
\end{align*}
The only option is therefore $\{t_{\mathrm{prec}}, t_{\mathrm{succ}}\}=\{x_{2i},x_{2i+2}\}$, and hence
$t=x_{2i+1}$, as this is the only vertex with label $-1$ that
commutes with $\{t_{\mathrm{prec}}, t_{\mathrm{succ}}\}=\{x_{2i},x_{2i+2}\}$.

\addtocounter{step}{1} {\bf Step~\thestep}. We claim that $d(x_{2j},x_{2t})=2$ implies
$x_{2i+1}x_{2j}\ne x_{2i+1}x_{2t}$ for each $i\in\ZZ_k$. Without loss of generality we assume $2t=2j+2$.

In fact, this claim is trivial whenever $x_{2i+1}$ is adjacent to both $x_{2j}$ and $x_{2t}$ because
then,  by \eqref{eq:start-even-cycle-thm}, $x_{2i+1}x_{2j}=x_{2j}\ne x_{2t}=x_{2i+1}x_{2t}$.

Suppose next $x_{2i+1}$ is adjacent to only one of $x_{2j}$ and $x_{2t}$. By symmetry we may assume that
$x_{2i+1}\sim x_{2j}\sim x_{2j+1}\sim x_{2j+2}=x_{2t}$. Let us denote $b:=x_{2i+1}$, $c:=x_{2i+2}$,
$d:=x_{2i+3}$ and $e:=x_{2i+4}=x_{2t}$. Suppose contrary to the claim that $bc=c=be$.
By~\eqref{eq:even-cycle},
    $bd,db\in\{d,b\}$. Now, $bd=d$ implies
    $e=de=(bd)e=b(de)=be=c$, a contradiction. Hence $bd=b$. Since $b\not\sim d$ we have $db=d$. However, this
    implies $c=dc=d(be)=(db)e=de=e$,
    a contradiction. Thus $bc=c$ and $be\ne c$.

Assume now the claim does not hold and let $x_{2i_0+1}$ and $x_{2j_0}$ be the vertices with least
distance for which
$$x_{2i+1}x_{2j}= x_{2i+1}x_{2t}; \qquad d(x_{2j},x_{2t})=2.$$
Clearly, $2t_0=2j_0+2$ or $2t_0=2j_0-2$. By the symmetry, we may assume the former, so
$\delta_0:=d(x_{2i_0+1},x_{2t_0})=d(x_{2i_0+1},x_{2j_0+2})>d(x_{2i_0+1},x_{2j_0})$.
 By the first part
of the proof of Step~\thestep{} we must have  $\delta_0\ge 5$. Then, there exists a vertex
$\hat{d}=x_{2j_0-1}=x_{2(j_0-1)+1}$ with label $-1$ (on the short arc from $x_{2i_0+1}$ to $x_{2j_0}$)
such that simultaneously $d(x_{2i_0+1},\hat{d})\le \delta_0-3$ and $d(\hat{d},x_{2j_0+2})=3$. By the
minimality of distance $\delta_0$ we have $\hat{d}x_{2i_0}\ne \hat{d}x_{2i_0+2}$, hence by
Step~\thestepp, $\hat{d}x_{2i_0+1}=\hat{d}$. Therefore,
\begin{align*}
\hat{d}x_{2j_0}&=(\hat{d}x_{2i_0+1})x_{2j_0}=\hat{d}(x_{2i_0+1}x_{2j_0})=\hat{d}(x_{2i_0+1}x_{2j_0+2})=
(\hat{d}x_{2i_0+1})x_{2j_0+2}\\
&=\hat{d}x_{2j_0+2},
\end{align*}
contradicting the fact that $\delta_0$ was the minimal distance with such equality possible, while
$d(\hat{d},x_{2j_0+2})=3<\delta_0$.

\addtocounter{step}{1} {\bf Step~\thestep}. Combining the previous step with \eqref{eq:x(2i+1)x(2j)} shows
that $x_1x_{2j}$ alternates between $x_0$ and $x_2$ as $j$ varies over $\ZZ_{k}$. This is only possible if $j$ is even, or equivalently, if the number of vertices in the
cycle is divisible by four.

Conversely, given a cycle with $4k$ vertices $C_{4k}=\{x_0,x_1,\dots,x_{4k-1}\}$, define the
multiplication in the only possible way (up to considering the opposite semigroup, i.e. up to
antiisomorphism), determined by \eqref{eq:start-even-cycle-thm}, \eqref{eq:x2ix(2i+2)=x(2i)},
\eqref{eq:ay=a}--\eqref{eq:11}, that is,
\begin{align*}
x_{2i}y&:=x_{2i};\\
  x_{2i+1}x_{2i+4j}&:=x_{2i};\quad
x_{2i+1}x_{2i+4j+2}:=x_{2i+2};\\
x_{2i+1}x_{2j+1}&:=x_{2i+1}
\end{align*} for all $i,j\in\ZZ_{2k}$ and $y\in
C_{4k}$,  where addition of subscripts in modulo $4k$.  It is straightforward to verify that this
multiplication is associative, hence makes $C_{4k}$ into a semigroup, and that
 $xy=yx$  if and only if $x$ and $y$ are adjacent vertices in
 the cycle $C_{4k}$. Moreover, as shown in previous steps, up to considering the opposite semigroup
this is the only option   for a centrefree semigroup whose commuting graph is the $4k$-cycle.
\end{proof}

\section{Commuting graph of a group}\label{sec:groups}
We now turn our attention to realizability over groups. Our main results classify graphs with an isolated
vertex or edge that are commuting graphs  of a group. Conversely, we also classify all groups whose
commuting graphs have an isolated vertex or edge.

Let us start with several results of a general nature. Our first lemma is well-known, but we give a
proof for the sake of completeness. Recall that $d(v)$ denotes the valency of a vertex $v$, that is, the
cardinality of the set of all vertices adjacent to $v$.
\begin{lemma} \label{centre-of-gp}
Let $G$ be a finite group with centre $Z$ and $\Gamma$ its commuting graph. Then  $|Z|$ is a common
divisor of the integers
$\{d(v)+1 \mid v\in\Gamma\}$.\\
 In particular, if $d_{\min}$ is the
minimal valency of vertices in $\Gamma$, then $Z$ has at most $d_{\min}+1$ elements.
\end{lemma}
\begin{proof}
Take any $v\in\Gamma$. Its neighborhood equals $\C_G(v)\setminus (Z\cup\{v\})$. Observe that $\C_G(v)$ is a
group which contains $Z$ as a subgroup and therefore, $|\C_G(v)|$ is divisible by $|Z|$. Hence, $|Z|$ also
divides $|\C_G(v)\backslash Z|=d(v)+1$.
\end{proof}

\begin{lemma}\label{lem:smalldiamconcomp}
Let $\Gamma_1$ be a connected component with diameter at most two of the commuting graph of a group
$G$. Then $\Gamma_1\cup Z(G)$ is a subgroup of $G$.
\end{lemma}
\begin{proof}
Choose any $v\in \Gamma_1$. Then $v^{-1}\in \C_G(v)$ and is clearly  not in $Z(G)$. So either $v$ is an
involution or $v^{-1}\sim v$, and hence in both cases we have that $v$ belongs to the connected component
containing $v$, that is, to $\Gamma_1$. It remains to show that the product of any $v,w\in Z(G)\cup\Gamma_1$
is also inside $Z(G)\cup \Gamma_1$. This is trivial if $v,w$ commute, since their product is either in
$Z(G)$ or else commutes with $v$ and $w$, hence is adjacent to both $v$ and $w$. Assume $v,w$ do not
commute. Since the diameter of $\Gamma_1$ is at most two, there exists $g\in\C_G(v)\cap\C_G(w)\cap\Gamma_1$.
Thus, $v,w\in\C_G(g)$, so also $vw\in\C_G(g)\subseteq Z(G)\cup \Gamma_1$.
\end{proof}

Our next example shows that the previous lemma does not hold for components with larger diameters.
\begin{example}
Let  $\Gamma=\Gamma(S_4)$.  It is an elementary calculation that the set of elements of order $2$ or $4$ form a
connected component of $\Gamma$ on $15$ vertices with diameter three. However, $S_4$ contains no
subgroups of order $15+|Z(S_4)|$.
\end{example}
A proper subgroup $M$ of a group $G$ is called a \emph{CC-subgroup} if $\C_G(m)\leqslant M$ for all $m\in
M\backslash\{1\}$. By Lemma \ref{lem:smalldiamconcomp}, if $Z(G)=1$ and $\Gamma_1$ is a connected component
with diameter at most two of the commuting graph of $G$ then $\Gamma_1\cup \{1\}$ is a CC-subgroup of $G$.
The structure of finite groups with a CC-subgroup was determined by Arad and Herfort \cite{AH}. Note
however, that there are groups with nontrivial centre whose commuting graph contains a connected component
of diameter at most two (see, for example Theorem~\ref{thm:edge-center}
below). Such a group cannot have a proper CC-subgroup.

 We now state two simple obstructions for realizability among groups. The second one
reflects sharply against the realizability with centrefree semigroups (c.f. Lemma~\ref{lem:cycle-semigps}).
We remark  that Theorems \ref{thm:isolated-vertex} and \ref{thm:14} give additional obstructions.

\begin{lemma}\label{lem:group-obstruction}
Let $\Gamma$ be a graph. If $|V(\Gamma)|\leq 4$ then $\Gamma$ is not the commuting graph of a group.
\end{lemma}
\begin{proof}
Let $\Gamma$ be a graph with at most four vertices and suppose that $\Gamma$ is the
     commuting graph of the group $G$. If there exists a vertex connected to every other vertex,
     the result follows from  Lemma~\ref{lem:obstruction}. Otherwise, the valency of each vertex is at
     most two. By Lemma~\ref{centre-of-gp} this implies that the centre of $G$ has at most three
     elements, and $|G|\leq 4+3=7$. All such groups are abelian except for $S_3$, but the
     commuting graph of $S_3$ has five vertices.
\end{proof}

\begin{lemma}\label{lem:edgeintriangle}
Let $\Gamma$ be the commuting graph of a group $G$. Then each edge $\{a,b\}$ that is not an isolated edge lies on a triangle.
\end{lemma}
\begin{proof}
Let $\{a,b\}$ be a non-isolated edge and let $\Gamma_1$ be the connected component containing $\{a,b\}$. Then either every vertex in $\Gamma_1\backslash \{a,b\}$ is at distance one from both $a$ and $b$ and so forms a triangle with $\{a,b\}$, or without loss of generality, there is some vertex $c$ with $a\sim b\sim c$ and $a\not\sim c$. Since $a$ and $b$
    commute,  their product, $ab$ commutes both with $a$ and with $b$. Since $a$ and $b$ are nontrivial and distinct, $\{a,ab,b\}$ are distinct, pairwise commuting elements.
Also, $ab\in Z(G)$ contradicts the fact that $c$ commutes with $b$ but not with $a=(ab)b^{-1}$. Thus
$ab\in G\setminus Z(G)$ forms a triangle with $\{a,b\}$.
\end{proof}

Recall that a bridge in a connected component $\Gamma_1$ of a graph $\Gamma$ is an edge whose deletion
(without removing  vertices) makes $\Gamma_1$  disconnected. Also,  a leaf is a vertex of valency one, thus  the unique edge containing it is a bridge.

\begin{corollary}\label{lem:bridge}
Let $\Gamma$ be  the commuting graph of a finite group $G$. If $\{u,v\}\in\Gamma$ is a bridge in some
connected component then $\{u,v\}$ is an isolated edge.
\end{corollary}

\begin{corollary}\label{cor:leaf}
Let $\Gamma$ be a commuting graph of a group and suppose that $u$ is a leaf of $\Gamma$ with unique
neighbour $v$. Then $\{u,v\}$ is an isolated edge.
\end{corollary}

Observe that  Lemma \ref{lem:edgeintriangle} and Corollary~\ref{cor:oddcycle} establish that the Petersen
graph is not the commuting graph of a group nor the commuting graph of a centrefree semigroup because it is
triangle free and contains a $5$-cycle as an induced subgraph.

The lexicographic product of graphs $\Gamma_1$ and $\Gamma_2$ is the graph $\Gamma_1[\Gamma_2]$, with vertex
set $V(\Gamma_1)\times V(\Gamma_2)$, where $(x_1,y_1)$ and $(x_2,y_2)$ form an edge if $(x_1,x_2)\in
E(\Gamma_1)$ or if $x_1=x_2$ and $(y_1,y_2)\in E(\Gamma_2)$. We note the following result of Vahidi and
Talebi~\cite{dihedral}.
\begin{lemma}\label{lem:lexico}\cite{dihedral}
Let $G$ be a group with nontrivial centre of size $k$. Then $\Gamma(G)$ is the lexicographic product
$\Gamma_1[K_k]$, where $\Gamma_1$ is the subgraph of $\Gamma$ induced on a set of representatives of
the nontrivial cosets of $Z(G)$ in $G$.
\end{lemma}

The structure of regular commuting graphs was essentially obtained by It\^{o} \cite{ito} and follows from
Lemma \ref{lem:lexico}. Recall  that a finite group whose order is a power of a prime $p$ is called a
$p$-group, and a finite group whose order is not divisible by a prime $p$ is called a $p'$-group.

\begin{lemma}\label{lem:regular}
Let $\Gamma$ be the commuting graph of a finite group $G$ and suppose that $\Gamma$ is regular. Then
$G\cong P\times A$ for some $p$-group $P$ and abelian $p'$-group $A$. Moreover, $\Gamma$ is the
lexicographic product $\Gamma_1[K_{p^a|A|}]$, where $|P|=p^{a+b}$ with $a,b\geq 1$ and $\Gamma_1$ has
order $p^b-1$.
\end{lemma}
\begin{proof}
Let $g\in G$ and let $d$ be the valency of $\Gamma$.  If $g\notin Z(G)$ then $g$ is a vertex of $\Gamma$ and
so has centraliser of order $|Z(G)|+d+1$. Thus the conjugacy classes of $G$ have size $|G|/(|Z(G)|+d+1)$ or
$1$ (for elements in $Z(G)$). The structure of groups with only two conjugacy class sizes was determined by
It\^{o} \cite{ito}, from which we deduce that $G$ is as in the statement of the lemma. Then $Z(G)=Z(P)\times
A$. The result  follows from the fact that  $p$-groups have nontrivial centre and Lemma \ref{lem:lexico}.
\end{proof}

We can now state the first main result of the present section. We say that a group automorphism $\phi$ is
\emph{fixed-point-free} if $x\phi=x$ implies $x=1$. If $x\phi=x$ for $x\ne1$ then $x$ is a nontrivial fixed
point. An automorphism of a group $A$ is referred to as \emph{inversion} if it maps each $a\in A$ to
$a^{-1}$.
\begin{theorem}\label{thm:isolated-vertex}
Suppose that $\Gamma$ is a finite graph with an isolated vertex. Then the following are equivalent.
\begin{enumerate}[label={\emph{(\roman*)}}]
\item $\Gamma$ is the commuting graph of a group.
\item $1\ne|V(\Gamma)|\equiv 1\pmod 4$  and $\Gamma$ has exactly $\frac{|V(\Gamma)|+1}{2}$ isolated
    vertices while the remaining $\frac{|V(\Gamma)|-1}{2}$ vertices form a complete graph.
\end{enumerate}
Moreover, the commuting graph of a group $G$ has an isolated vertex if and only if $G$ is the semidirect
product $A\rtimes \ZZ_2$ for some abelian group $A$ of odd order where $\ZZ_2$ acts on $A$ by inversion.
\end{theorem}
\begin{proof}
Let $\Gamma$ be a finite graph with isolated vertex $v$ and suppose that $\Gamma$ is the commuting graph of a
group $G$.
 By Lemma~\ref{centre-of-gp}, $|Z(G)|=1$ and so $G$ is a finite group with $V(\Gamma)=G\backslash\{1\}$.
Since $v$ commutes with $v^2$ but $v$ has no neighbours,  we must have $v^2\in Z(G)=1$.  As conjugation by
an element of $G$ induces an automorphism of $\Gamma$,   every element in the conjugacy class of $v$
corresponds to an isolated vertex in $\Gamma$. Since the conjugacy class of $v$ contains $|G|/|\C_G(v)|$
elements,  and since $\C_G(v)=\{1,v\}$ we see that $|G|=2n$ is even and at least $n$ vertices are isolated.
We label them so that the conjugacy class of $v$ consists of $\{v_1,\dots,v_n\}$ where $v_1=v$. In
particular, $v_i^2=1$ for each~$i$.

Now, for distinct indices $i,j$, if $v_iv_j$ is isolated then as above $(v_iv_j)^2=1$, which we
rewrite as  $v_iv_j=v_j^{-1}v_i^{-1}=v_jv_i$, so $v_i$ and $v_j$ commute contradicting the fact that
$v_i$ is isolated. Hence, fixing $i$, we have that $\{v_iv_1,\dots,v_iv_n\}\backslash\{v_i^2\}$
consists of  $n-1$ pairwise distinct and nonisolated vertices. Since $|G\backslash\{v_1,\dots,v_n,1\}|=n-1$, it follows that every vertex in  the set $\hat{N}:=G\backslash\{v_1,\dots,v_n,1\}$ can be written as $v_iv_j$ for some $j$.
Let $N:=\hat{N}\cup\{1\}$. Then for $u,w\in N$ we have $u=v_1v_i$ for some $i$ and $w=v_iv_k$ for some $k$
and hence $uw=v_1v_iv_iv_k=v_1v_k\in N$. Also, $u^{-1}=v_iv_1\in N$, so $N$ is a subgroup of index $2$ in
$G$ and therefore, a normal subgroup.

The map $n\mapsto v^{-1}nv$ is an automorphism of $N$ of order two. Moreover, it is
fixed-point-free   because $\C_G(v)=\{1,v\}$. This implies (see \cite[Theorem 1.4, p.~336]{Gorenstein}) that
$N$ is  abelian   and $v^{-1}nv=n^{-1}$ for all $n\in N$. Therefore, $|N|=n=2k+1$ must be odd (otherwise,
$n^{-1}=n$ for some $n\in N\backslash\{1\}$). Hence $|V(\Gamma)|=|G|-1=2n-1=4k+1$ for some integer $k$, and
$\Gamma$ has exactly $n=\frac{|V(\Gamma)|+1}{2}$ isolated vertices, while the remaining
$|N|-1=\frac{|V(\Gamma)|-1}{2}$ vertices  lie in  the abelian subgroup $N$, so they form a complete induced
subgraph, as claimed in (ii).  Moreover, since $v_1\notin N$ and has order two, it follows that $G=N\rtimes \langle v_1\rangle$.

Conversely, if $A$ is an abelian group of odd order $n=2k+1\ge3$, then $A$ has no
elements of order two, so inversion is  a fixed-point-free automorphism of order two.
This implies that $G:=A\rtimes \ZZ_2$, where $\ZZ_2$ acts on $A$ by inversion, has trivial
centre. It is easily seen that the commuting graph $\Gamma(G)$ has an isolated vertex, corresponding
to the generator of $\ZZ_2$, and by the first part of the proof, $\Gamma(G)$ is a graph with the
properties stated in (ii). This also shows that (ii) implies (i), because if $\Gamma$ is a graph with
properties as in (ii) of order
 $4n+1\ge 5$, then there  exists an abelian group $A$ of odd order  $2n+1$.
\end{proof}

\begin{corollary}
If $\Gamma$ has an isolated vertex and $|V(\Gamma)|=2p-1$  for some odd prime $p$ then there exists,
up to isomorphism, exactly one group such that $\Gamma$ is its commuting graph.
\end{corollary}

We also have the following corollary which will be useful later.
\begin{corollary}\label{cor:fixed-pts}
If $N$ is a nontrivial group and $\phi\colon S_3\to \mathrm{Aut}(N)$ is a group homomorphism, then there
exists $a\in S_3\backslash\{1\}$ such that $(a\phi)$ has a nontrivial fixed point in $N$.

Moreover, if $v\in S_3$ is a $3$-cycle and $(v\phi)\in\mathrm{Aut}(N)$ is fixed-point-free, then for
every involution $a\in S_3\backslash\{1\}$, $(a\phi)$ has a nontrivial fixed point in~$N$.
\end{corollary}
\begin{proof}
Form the semidirect product $G:=N\rtimes_\phi S_3 $. If $v=(123)\in S_3$ is a $3$-cycle and $v\phi$  is a fixed-point-free  automorphism of $N$, then $\C_G(v)=\{1,v,v^2\}$ in which case the commuting
graph $\Gamma$ of $G$
 contains an isolated edge
$\{v,v^2\}$. Choose any involution $g\in S_3$. If $(g\phi)$ is  also a fixed-point-free automorphism, then
 $\C_G(g)=\{1,g\}$, so $\{g\}$ is an isolated vertex in $\Gamma$. Hence, $\Gamma$ contains an
isolated vertex and an isolated edge. It easily follows from Theorem~\ref{thm:isolated-vertex} that $G=S_3$,
a contradiction.
\end{proof}

We next study  groups whose commuting graph has an isolated edge. We will not be able to give as precise a
description of the groups in this case as we are only aware of a structural description of  nilpotent groups
admitting fixed-point-free automorphisms of order three, rather than a complete classification.

\begin{lemma}\label{lem:centre-in-case-of-an-edge}
Let $G$ be a finite group  whose commuting graph $\Gamma=\Gamma(G)$ contains an isolated edge $\{v,w\}$.
Then $Z(G)$ contains at most two elements. Moreover: \medskip\\
(i) If $|Z(G)|=1$ then $|v|=3$. Also, $w=v ^2$ and $\C_G(v )=\langle v\rangle$.\smallskip\\
(ii) If $|Z(G)|=2$ then either

\begin{enumerate}[label=\emph{(\alph*)}]
 \item  $|v|=4$. Also, $ Z(G)=\{1,v^2\}$, $w=v^3$, and  $\C_G(v)=\langle v\rangle\cong\ZZ_4$, or
 \item $|v|=2$. Also, $ Z(G)=\{1,z\}$, $w=zv$,  and $\C_G(v)=\langle v,z\rangle\cong\ZZ_2\times\ZZ_2$.
\end{enumerate}
\end{lemma}
\begin{proof}
Each vertex in an isolated edge has valency one, so the claim about the size of the centre of $G$
follows from Lemma~\ref{centre-of-gp}. Assume first $Z(G)$ is trivial. Then $\langle v\rangle
\leqslant \C_G(v)=\{v,w\}\cup Z(G)=\{1,v,w\}$. Thus $v $ has order three, and (i) follows.

Assume next that $Z(G)=\{1,z\}$ has two elements. Then $z^2=1$ and
 $|G|=|V(\Gamma)|+2$ and
clearly, $zv$ is noncentral, but commutes with $v$. Hence, $w=zv$. Also, $v^2$ commutes with $v$, so
either $v^2=w=zv$ or $v^2\in Z(G)$. The former possibility contradicts the fact that $v$ is
noncentral, so
$$v^2\in\{1,z\}.$$
If $v^2=1$ then $\C_G(v)=\{1,v,z,zv\}\cong\ZZ_2\times\ZZ_2$. However, if  $v^2=z$ then
$\C_G(v)=\{1,v,v^2,v^3\}\cong\ZZ_4$ and so $w=v^3=zv$.
\end{proof}

We can now combine Lemma \ref{lem:centre-in-case-of-an-edge} with several group-theoretic results to
obtain a characterisation of centrefree groups whose commuting graph has an isolated edge. Note that,
as additive groups, $\ZZ_2^4$ is isomorphic to the  $\GF(4)^2$ and thus has a natural $\SL(2,4)$ module
structure.
\begin{lemma}\label{lem:Gtrivialcentre}
Let $G$ be a finite group with trivial centre whose commuting graph has an isolated edge. Then one of
the following holds:
\begin{enumerate}[label={\emph{(\roman*)}}]
 \item  $G\cong N\rtimes \ZZ_3$ or $G\cong N\rtimes S_3$ where $N$ is a nilpotent group of
     nilpotency class at most two with $|N|\equiv 1\pmod 3$, and each element of $G$ of order three
     acts as a fixed-point-free automorphism of~$N$.
 \item $G\cong N\rtimes A_5$ where $N$ is the direct product of copies of $\ZZ_2^4$, each viewed as the
     natural module for $\SL(2,4)\cong A_5$.

 \item $G\cong \PSL(2,7)$.
\end{enumerate}
\end{lemma}
\begin{proof}
Let $\{v,w\}$ be an isolated edge. Then by Lemma \ref{lem:centre-in-case-of-an-edge}, $\langle
v\rangle$ is a self-centralising subgroup of $G$ of order three. Thus by Feit and
Thompson~\cite{Feit-Thompson} we have one of the following:
\begin{enumerate}
 \item[(a)] $G$ has a normal nilpotent subgroup $N$ such that $G/N\cong \ZZ_3$ or $S_3$.
 \item[(b)] $G$ has a normal 2-subgroup $N$ such that $G/N\cong A_5$.
 \item[(c)] $G\cong \PSL(2,7)$.
\end{enumerate}
Note that $\C_G(v)=\langle v\rangle\cong\ZZ_3$ implies that its normalizer, $\N_G(\langle v\rangle)$
either fixes elements in $\C_G(v)$ or swaps $v$ and $v^2$, so  $\N_G(\langle v\rangle)\cong \ZZ_3$ or
$S_3$.

In case (a) we only need to consider $N\ne1$. Then $N$ has a nontrivial centre and as $\langle v\rangle$ is
self-centralising we must have either $N=\langle v\rangle$ or $N\cap \langle v\rangle=1$. The first option
is not possible, as it would imply that a Sylow 3-subgroup of $G$ has order at least 9, contradicting the
fact that $\langle v\rangle$ is self-centralising. Thus $N\cap \langle v\rangle=1$ and so conjugation by $v$
is a fixed-point-free automorphism of $N$. Then, considering the orbits of this action, $|N|\equiv 1\pmod
3$. Consequently, $\langle v\rangle$ is a Sylow 3-subgroup of $G$ and hence all elements of $G$ of order
three are conjugate to either $v$ or $v^{-1}$, and so act fixed-point-freely on $N$. Moreover, as
$\N_G(\langle v\rangle)\cap N\norml \N_G(\langle v\rangle)$ and $v\notin N$ it follows that $\N_G(\langle
v\rangle)\cap N =1$. By Mazurov, \cite[Theorem p.29]{Mazurov} we have that $G=N\N_G(\langle v\rangle)$ and
so either $G\cong N\rtimes \ZZ_3$ or $G\cong N\rtimes S_3$. The additional structure for $N$ and $G$ in
cases (a) and (b) follows from Mazurov \cite[Theorem p. 29 and Lemma 9 p. 33]{Mazurov}.
\end{proof}

We note that not every group appearing in (a) and (b) of the proof of Lemma \ref{lem:Gtrivialcentre} has
a self-centralising subgroup of order three.
 For example, the dihedral group $D_{12}$ of order $12$ contains a normal cyclic (hence nilpotent) subgroup $N$ of order two such that
$G/N$ is isomorphic to $S_3$. However, the unique subgroup of order three in $D_{12}$ is not
self-centralising.

We now investigate the structure of commuting graphs in each of the cases given by Lemma
\ref{lem:Gtrivialcentre}. We denote by $nK_i+m K_j$ a disjoint union of $n$ copies of the complete graph
$K_i$ and $m$ copies of the complete graph $K_j$.

\begin{theorem}\label{thm:14}
Let $\Gamma$ be the commuting graph of a finite group $G$ with trivial centre and suppose that $\Gamma$  has
an isolated edge. Then one of the following holds:
\begin{enumerate}[label=\emph{(\roman*)}]
\item\label{triv1} $G\cong S_3$ and  $\Gamma=3K_1+K_2$.
\item $G\cong A_5$ and $\Gamma=10K_2+5K_3+6K_4$.
\item $G\cong \PSL(2,7)$ and $\Gamma=28 K_2+8 K_6+\Delta$, where $\Delta$ is a connected component
    on $63$ vertices with diameter five.
 \item\label{triv4}  $G\cong N\rtimes \ZZ_3$ and $\Gamma=(\frac{|V(\Gamma)|+1}{3})\,K_2+\Delta$,
     where $\Delta$ is a  connected component of size $\frac{|V(\Gamma)|-2}{3}$ containing a
     vertex adjacent to all other vertices in $\Delta$.
\item  $G\cong N\rtimes S_3$ and $\Gamma=(\frac{|V(\Gamma)|+1}{6}) \, K_2+\Delta$, where $\Delta$
    is a connected component of diameter three.
\item $G\cong N\rtimes A_5$ and $\Gamma=(\frac{|V(\Gamma)|+1}{6})\, K_2+
    (\frac{|V(\Gamma)|+1}{10})\, K_4+ \Delta$, where $\Delta$ is a connected component of diameter
    three which contains a clique $C$ of size $\frac{|V(\Gamma)|+1}{60}$ and each element of $\Delta$
    is adjacent to an element of~$C$.
\end{enumerate}
In the last three cases the structure of $N$ is given in Lemma \ref{lem:Gtrivialcentre}.
\end{theorem}
\begin{proof}
Let $\{v,w\}$ be an isolated edge. By Lemma \ref{lem:centre-in-case-of-an-edge}, $|v|=3$ and $w=v^2$.
The possibilities for $G$ are listed in Lemma \ref{lem:Gtrivialcentre}. If $G\cong S_3, A_5$ or
$\PSL(2,7)$ then we have one of the first three cases.

Suppose that $G\cong N\rtimes \ZZ_3$ or $N\rtimes S_3$ for some nontrivial nilpotent group $N$ with
properties as in (1) of Lemma  \ref{lem:Gtrivialcentre}. By Lemma \ref{lem:Gtrivialcentre}, we know that
$\C_G(v)=\langle v\rangle\cong \ZZ_3$ is a Sylow 3-subgroup of $G$   and so there are precisely
$$|G|/|\N_G(\langle v\rangle)|$$ isolated edges. Moreover, $\N_G(\langle v\rangle)\cong \ZZ_3$ or $S_3$ and
$\N_G(\langle v\rangle)\cap N=1$. Thus $G=N\rtimes \N_G(\langle v\rangle)$ and the number of isolated edges
equals $|G|/|\N_G(\langle v\rangle)|=(|V(\Gamma)|+1)/3$ when $G/N\cong \ZZ_3$, and $|G|/|\N_G(\langle
v\rangle)|=(|V(\Gamma)|+1)/6$ when $G/N\cong S_3$.

If $G/N\cong \ZZ_3$, then these cover all the elements of $G$ not in $N$. Since $N$ is nilpotent, it has a
nontrivial centre and so the elements of $N\backslash \{1\}$ all have at least one common neighbour. This
gives us case \emph{\ref{triv4}}.

Let $G/N\cong S_3$. Conjugation by $v$ acts fixed-point-freely on $N$ hence also  on $Z(N)$, so Corollary
\ref{cor:fixed-pts} implies that  each involution $g\in \N_G(\langle v\rangle)\cong S_3$ centralises a
nontrivial element of $Z(N)$. Thus each $ng\in Ng$ centralises some $z\in Z(N)\backslash\{1\}$ and  so
$$\Delta=\big(N\cup
\bigcup_{g\in \N_G(\langle v\rangle)\atop |g|=2}Ng\big)\setminus\{1\}$$
 forms a
connected component of diameter at most three with  $|\Delta|=4|N|-1=4\frac{|V(\Gamma)|+1}{6}-1$. Let
$g_1,g_2$ be distinct involutions in $\N_G(\langle v\rangle)$ and suppose that there exists $x\in
\C_G(g_1)\cap \C_G(g_2)$. Then $x=nh$ for some $n\in N$ and $h\in \N_G(\langle v\rangle)$. Now
$nh=(nh)^{g_i}=n^{g_i}h^{g_i}$ with $n^{g_i}\in N$ and $h^{g_i}\in \N_G(\langle v\rangle)$. Thus both $h$
and $n$ are centralised by $g_1$ and $g_2$. Hence $h=1$ and $n$ is centralised by $g_1g_2$. However,
$g_1g_2$ is a nontrivial element of $\langle v\rangle$, which acts fixed-point-freely on $N$, a
contradiction. Thus the diameter of $\Delta$ is three.

Finally suppose that $G\cong N\rtimes H$ with  $H\cong A_5\cong\SL(2,4)$ and with $N=\bigoplus_1^t\ZZ_2^4$
the direct sum of $t$ copies of the natural module $\ZZ_2^4\cong \GF(4)^2$  for $\SL(2,4)$. Then $\langle
v\rangle$ is a Sylow 3-subgroup of $G$ and so there are precisely $|G|/|\N_G(\langle
v\rangle)|=(|V(\Gamma)|+1)/6$ isolated edges (use that $S_3\subseteq H$, so $v$ and $v^2$ are conjugate).
Let $g\in G$ have order five. By Sylow's Theorem, $g$ is conjugate to an element of $H$, so we may assume
that $g\in H$. Since no matrix of order five in $\SL(2,4)$ acting on its natural module has 1 as an
eigenvalue it follows that $\C_G(g)\cap N=1$. Thus $\C_G(g)\cong \C_G(g)N/N\leqslant \C_{G/N}(gN)\cong
\C_H(g)=\langle g\rangle$. Hence $\C_G(g)=\langle g\rangle$ and $\langle g\rangle \backslash\{1\}$ is a
connected component of $\Gamma(G)$ isomorphic to $K_4$. Since $\N_G(\langle g\rangle)\cap N$ is a normal
subgroup of $\N_G(\langle g\rangle)$ that is disjoint from $\langle g\rangle$ it must centralise $\langle
g\rangle$ and so $\N_G(\langle g\rangle)\cap N=1$. Thus $\N_G(\langle g\rangle)\cong N\N_G(\langle
g\rangle)/N$, which is isomorphic to a subgroup of $\N_{G/N}(\langle Ng\rangle)\cong \N_H(\langle
g\rangle)\cong D_{10}$. Hence  $\N_G(\langle g\rangle)=\N_H(\langle g\rangle)\cong D_{10}$, which gives
 $\frac{|G|}{|{\mathcal
N}_G(\langle g\rangle)|}=\frac{|G|}{10}$ isolated copies of $K_4$ in $\Gamma$.
 It remains to consider the set
$$\Delta=(N\cup\bigcup_{g_i\in H\atop |g_i|=2} g_i N)\setminus\{1\}$$
which has size $|G|-2|G|/6-4|G|/10-1=16|G|/60-1$
 (note that $A_5$ contains 15
involutions). Since $N$ is abelian, the elements of $N\backslash \{1\}$ form a clique.  Moreover, since
$|N|$ is even, each involution in $H$ centralises a nontrivial element of $N$ and so every element of
$\Delta$ is adjacent to some element of $N\backslash\{1\}$. Hence the elements of $\Delta$ form a single
connected component of diameter at most three.
 To show
that its diameter is at   least three, choose involutions $g_1,g_2\in \N_G(\langle g\rangle)\cong
D_{10}$. Since $\langle g_1,g_2\rangle=\N_G(\langle g\rangle)$ and $\C_N(g)=1$, it follows that $g_1$
and $g_2$ do not centralise the same nontrivial element of $N$. Consequently, if $g_1\sim x\sim g_2$
is a path in $\Gamma(G)$, then $x\notin N$ and as such, $g_1N\sim xN\sim g_2N$ would be a path  in
$\Gamma(G/N)=\Gamma(A_5)$. However, in $\Gamma(A_5)$ the involutions lie in cliques of size three,
contradicting the fact that $\langle g_1N,g_2N\rangle \cong \langle g_1,g_2\rangle \cong D_{10}$ is
nonabelian. Hence $\Delta$ has diameter three.
\end{proof}

We now investigate groups with nontrivial centre. First we define some groups:
\begin{de}\label{def:groups}
\mbox{}
\begin{enumerate}[label={(\roman*)}]
\item Let $J=\langle a,b,c,\gamma\mid a^3=b^3=c^2=abc=\gamma^2, a^\gamma=b\rangle$ be a non-split
    extension of $\SL(2,3)=\langle a,b,c\mid a^3=b^3=c^2=abc\rangle$ by $\langle\gamma\rangle
    \cong\ZZ_4$; This is SmallGroup(48,28) in GAP~\cite{GAP4}. We refer to~\cite[p.
    104--105]{Wong2} for a realization  of $J$ as the subgroup of  semilinear transformations
    $\Gamma\mathrm{L}(2,9)$.
\item Let $D_{2n}=\langle a,b \mid a^n=b^2=1, a^b=a^{-1} \rangle $  for $n\geq 2$ be a dihedral group.
    We remark that $D_4\cong \ZZ_2\times \ZZ_2$;
\item Let $SD_{2n}=\langle a,b\mid a^{n}=b^2=1,a^b=a^{ n/2-1} \rangle$ for $n=2^k\geq 8$ be a
    semidihedral group,  and
\item Let $Q_{4n}=\langle a,b \mid a^{2n} = b^4 = 1, a^n = b^2, a^b = a^{-1}\rangle$ for $n\geq 2$
    be a generalised  quaternion group.
\end{enumerate}
\end{de}

\begin{lemma}\label{lem:edgenontrivcent}
Let $G$ be a finite group with nontrivial centre whose commuting graph contains an isolated edge
$\{v,w\}$. Then one of the following holds:
\begin{enumerate}[label=\emph{(\roman*)}]
 \item $G\cong\SL(2,3)$ or $G\cong \SL(2,5)$, or
 \item $G$ has an abelian normal subgroup $N$ of odd order with $G/N\cong
     \GL(2,3)$ or $J$, the preimage of $\SL(2,3)$ in $G$ centralises $N$, and $v$ acts on $N$ by inversion, or
\item $G\cong N\rtimes H$ where $N$ is an abelian group of odd order and $H$ is isomorphic to one of
    $\ZZ_4$, $D_4$, $D_8$, $Q_8$, or $D_{2^k}$, $SD_{2^k}$  $Q_{2^k}$ with \mbox{$k\ge 4$}. Furthermore,
    the group induced by the action of $H$ on $N$ via conjugation is $\ZZ_2$ or $\ZZ_2\times \ZZ_2$, and
    $v$ acts on $N$ by inversion.
\end{enumerate}
\end{lemma}
\begin{proof}
Let $\{v,w\}$ be an isolated edge in the commuting graph $\Gamma(G)$. By Lemma
\ref{lem:centre-in-case-of-an-edge}, $|Z(G)|=2$ and either $|v|=4$ with $w=v^3$ and $Z(G)=\langle
v^2\rangle$, or $|v|=2=|w|=|vw|$ with $Z(G)=\langle vw\rangle$. Thus, $\C_G(v)=\C_G(w)=\C_G(\langle
v,w\rangle)=\langle v,w\rangle$ and $\langle v,w\rangle$ is a self-centralising subgroup of $G$ of order
four.
Let $N$ be the largest odd order normal subgroup of $G$. Since
$|Z(G)|=2$ we have $Z(G)\cap N=1$ and
so $Z(G/N)\geq 2$. Thus by  Wong \cite[Theorems 1 and 2]{Wong}, $G/N$ is isomorphic to one of $\ZZ_4$,
$D_4$, $D_8$, $Q_8$, $\SL(2,3)$, $\SL(2,5)$, $\GL(2,3)$,  $J$, or $D_{2^k}$, $SD_{2^k}$ or $Q_{2^k}$ for
some $k\ge 4$.

Since $|\C_G(v)|=4$ while $|N|$ is odd we have that $\C_G(v)\cap N=1$. Combined with $v^2\in Z(G)$, it
follows that $v$ induces a fixed-point-free automorphism of $N$ of order two. This implies (see
\cite[Theorem 1.4, p.~336]{Gorenstein}) that $N$ is  abelian and $v$ acts on $N$ by inversion.

Suppose that $N\neq 1$ and let $\phi\colon G\rightarrow \Aut(N)$ be the homomorphism induced by the action
of $G$ on $N$ by conjugation. Let  $M=\ker \phi$.    Then $\ZZ_2\cong Z(G)\leqslant M$ and $Z(G)\cap N=1$.
Thus $N<M\norml G$ and  $M/N$ contains a central subgroup of $G/N$ or order two. Moreover,
$v\notin\ker(\phi)$ and since  $v\phi$ is inversion we have that $1\neq v\phi\in Z((G)\phi)\cong Z(G/M)$. In
particular, letting $R$ be the preimage of $Z((G)\phi)$ in $G$ we have the chain of subgroups
$$N<NZ(G)\leqslant M<R\leqslant G$$
each normal in $G$, with $R/M=Z(G/M)\neq 1$ and $NZ(G)/N$  being central in $G/N$ and of order two. This is
impossible when $G/N\cong \SL(2,3)$ or $\SL(2,5)$. Thus, if  $G/N\cong \SL(2,3)$ or $\SL(2,5)$ then $N=1$
and (i) holds.

Assume next that $G/N\cong\GL(2,3)$ or $J$. Since $\GL(2,3)$ and $J$ have a unique normal subgroup of order
two we must have $G/NZ(G)\cong S_4$. Using the normal structure of $S_4$ and the fact that $R/M$ is a
nontrivial central subgroup of $G/M$, it follows that $R=G$ and $M/NZ(G)\cong A_4$. In particular, $M/N\cong
\SL(2,3)$ and by definition $M$ centralises $N$. Thus $(ii)$ holds.

Suppose finally that $G/N\cong\ZZ_4$, $D_4$, $D_8$, $Q_8$, or $D_{2^k}$, $SD_{2^k}$ or $Q_{2^k}$ with $k\ge
4$. Since $N$ is odd, Sylow's Theorems imply that $G=N\rtimes H$ for some Sylow 2-subgroup $H$ of $G$
containing $v$. Clearly, $H\cong G/N$. Now $v$ acts on $N$ by inversion and by
Lemma~\ref{lem:centre-in-case-of-an-edge}, $|\C_G(v)|=4$. Thus if $H\cong \ZZ_4$ or $D_4$, then $H=\C_G(v)$
and $(H)\phi=\ZZ_2$ as $Z(G)\leqslant H$ and acts trivially on $N$. In the rest of the cases, take standard
generators $a$ and $b$ for $H$ as given in Definition \ref{def:groups}, with $b$ having order four when
$H\cong Q_{2^k}$ and order two otherwise. Suppose that $|H| \ge 16$. Since $|\C_G(v)|=4$ it follows that
$v=a^ib$ for some integer $i$. Moreover, as $a^v=a^{-1}$ or $a^{2^{k-2}-1}$ but the inversion map $(v)\phi$
commutes with $(a)\phi$, it follows that $(a^2)\phi=1$. Hence $(H)\phi=\ZZ_2$ or $\ZZ_2\times\ZZ_2$.  The
same argument holds when $|H|=8$ and $v=a^ib$. Thus it remains to consider the case where $H=D_8$ or $Q_8$,
and $v=a$. However, since $v^b=v^{-1}$ and $(v)\phi$ commutes with $(b)\phi$ it once again follows that
$(a^2)\phi=1$. Thus (iii) holds.
\end{proof}

We now determine the graphs that arise from the groups listed in  Lemma \ref{lem:edgenontrivcent}.
Commuting graphs of dihedral groups and generalised quaternion groups were studied by \cite{dihedral}.

\begin{theorem}\label{thm:edge-center}
Let  $\Gamma$ be the commuting graph of a group $G$ with nontrivial centre and suppose that $\Gamma$
has an isolated edge. Then one of the following holds:
\begin{enumerate}[label=\emph{(\roman*)}]
  \item\label{sl23}$G\cong \SL(2,3)$ and  $\Gamma(G)=3K_2+4K_4$,
\item\label{sl25} $G\cong \SL(2,5)$ and $\Gamma(G)=15 K_2+10K_4+6K_8$,
\item\label{gl23}$G\cong \GL(2,3)$ or $J$, and  $\Gamma(G)=6K_2+4K_4+3K_6$,
\item\label{G8} $G\cong Q_8$ or $D_8$ and $\Gamma(G)=3K_2$.
\item\label{gendih}$G\cong D_{2^k}$, $SD_{2^k}$ or $Q_{2^k}$ with $k\ge 4$, and   $\Gamma(G)= 2^{k-2}
    K_2+ K_{2^{k-1}-2}$

\item $G$ is as in part \emph{(ii)} of Lemma \ref{lem:edgenontrivcent} with $N\neq 1$ and
    $\Gamma(G)=\linebreak\left(\frac{|V(\Gamma)|+2}{8}\right)K_2 +\Delta$, with $\Delta$  connected of
     diameter four.

\item\label{zz4}\label{phieq2}\label{phieq22} $G$ is as in part \emph{(iii)} of Lemma \ref{lem:edgenontrivcent} with $N\neq 1$ and either
  \begin{enumerate}
     \item[\emph{(a)}] $\Gamma(G)=\left(\frac{|V(\Gamma)|+2}{4}\right)K_2+K_k$ where $k=(|V(\Gamma)|-2)/2$, or
     \item[\emph{(b)}] $\Gamma(G)=\frac{|V(\Gamma)| +2}{8} K_2+\Delta$, with  $\Delta$ connected of diameter at most three.
  \end{enumerate}
\end{enumerate}
\end{theorem}
\begin{proof}
 Let $\{v,w\}$ be an isolated edge. By Lemma \ref{lem:centre-in-case-of-an-edge}, $|Z(G)|=2$ and  $\langle v,w\rangle\cong \ZZ_4$ or $\ZZ_2\times\ZZ_2$. The possibilities for $G$ are given in Lemma \ref{lem:edgenontrivcent}. If $G\cong \SL(2,3)$ we have case \emph{\ref{sl23}}, if  $G=\SL(2,5)$ we have case \emph{\ref{sl25}}, while if $G=\GL(2,3)$ or $J$ we have case \emph{\ref{gl23}}.

To continue, suppose first that $G/N\cong \GL(2,3)$ or $J$, where $N$ is a nontrivial abelian normal
subgroup of odd order. We will use various properties of $\GL(2,3)$ and $J$ that can be verified using
\textsc{Magma} \cite{magma} or GAP \cite{GAP4}. By Lemma \ref{lem:edgenontrivcent}, the preimage $H$ of
$\SL(2,3)$ in $G$ centralises $N$ and $v$ acts on $N$ by inversion. Since $|Z(G)\cap \langle v,w\rangle|=2$,
any element of $G$ that normalises $\langle v,w\rangle$ either centralises $\langle v,w\rangle$ or
interchanges $v$ and $w$. Thus $|\N_G(\langle v,w\rangle)|=4$ or 8.  A Sylow 2-subgroup of $G$ is isomorphic
to a Sylow 2-subgroup  of $G/N$, and so looking in $\GL(2,3)$ or $J$, we see that $|\N_G(\langle
v,w\rangle)|=8$. Moreover, in $\GL(2,3)$ and $J$ all self-centralising subgroups of order 4 are conjugate
and so $G$  has only one conjugacy class of self-centralising subgroups of order 4. Thus there are $|G|/8$
isolated edges in $\Gamma$. Since $N$ is abelian, the elements of $NZ(G)\backslash Z(G)$ form a clique and
each element of $H\backslash NZ(G)$ is adjacent to each element of $N$. As $H/N\cong \SL(2,3)$ is nonabelian
the graph induced on $H\backslash Z(G)$ is of diameter two and contains $|G|/2-2$ vertices. If
$g\in\GL(2,3)\backslash\SL(2,3)$ (respectively $J\backslash \SL(2,3)$) does not lie in a self-centralising
subgroup of order $4$ then $g$ has order $8$, $g^2\in \SL(2,3)$ and $\C_{\GL(2,3)}(g^2)=\langle g\rangle $
(respectively $\C_J(g^2)=\langle g\rangle$).  Hence, given two elements $g_1,g_2\in G\backslash H$ that are
not in an isolated edge we have that $g_1^2,g_2^2\in H\backslash Z(G)$ and so for arbitrary $n\in
N\backslash \{1\}$ we have that $g_1\sim g_1^2\sim n\sim g_2^2\sim g_2$ is a path of length four
in~$\Gamma$. Choose $g_1,g_2$  so that $\langle g_1N\rangle$, $\langle g_2N\rangle$ are distinct
self-centralizing subgroups of order eight in $G/N$ whose intersection is $Z(G/N)$.  If  $g_1\sim a\sim
b\sim g_2$ were a path of length three in $\Gamma$ then $g_1N\sim aN\sim bN \sim g_2N$  is a path in the
commuting graph of  $G/N$ where we now allow elements in the centre. Since $\langle g_1N\rangle$ and
$\langle g_2N\rangle$ are self-centralising in $G/N$ it follows that $aN\in\langle g_1N\rangle$ and
$bN\in\langle g_2N\rangle$. Since the only elements of $\langle g_1N\rangle$ which commute with elements not
in $\langle g_1N\rangle$ are those in $Z(G/N)$ it follows that $aN\in Z(G/N)$, contradicting the fact that
$g_1\in G\backslash H=vH$ acts on $N$ as a fixed point free inversion. Thus the set of vertices of $\Gamma$
not in an isolated  edge forms a connected subgraph of diameter four and $\Gamma$ is as in part (vi).

Next suppose  that $G=N\rtimes \ZZ_4$ or $N\rtimes D_4$, with $N$ a nontrivial abelian group of odd order.
Since $|Z(G)\cap \langle v,w\rangle|=2$, any element of $G$ that normalises $\langle v,w\rangle$ either
centralises $\langle v,w\rangle$ or interchanges $v$ and $w$. Since $\langle v,w\rangle$ is
self-centralising and is a  Sylow 2-subgroup of $G$ it follows that $\N_G(\langle v,w\rangle)=\langle
v,w\rangle$ and there are $|G|/4=(|V(\Gamma)|+2)/4$ isolated edges in $\Gamma$. This covers $|G|/2$ of the
vertices in $G$ and consists of all elements not in the index two  normal subgroup $NZ(G)$. Since $N$ is abelian, so is
$NZ(G)$ and hence the vertices not in an isolated edge form a clique of size $(|V(\Gamma)|-2)/2$.
 This gives case (vii)(a).

Next suppose that $G=N\rtimes H$ with $H\cong D_{2^k}$,  $SD_{2^k}$ or $Q_{2^k}$ and $k\geq 3$ (if $H\cong
SD_{2^k}$ then $k\ge 4$). Take standard generators $a$ and $b$ for $H$ as in Definition \ref{def:groups}.
 Let
$z=a^{2^{k-2}}$ such that $Z(G)=\langle z\rangle$.
 If $N=1$ and $k=3$ then $G\cong D_8$ or $Q_8$ and $\Gamma(G)=3K_2$. This is case \emph{\ref{G8}}. If $N=1$ and $k\geq 4$, then  the elements of $\langle a\rangle\backslash \langle z\rangle$ form a clique of size $2^{k-1}-2$. Elements of the form $a^ib$ have order two when $H$ is dihedral and order four when $H$ is quaternion. When $H$ is semidihedral $a^ib$ has order two when $i$ is even and order four when $i$ is odd. In all three cases, $\C_H(a^ib)=\langle a^ib,z\rangle$ has order four.
 Thus $\Gamma(G)$ is as in part \emph{\ref{gendih}}.

Suppose now that $N\neq 1$. By Sylow's Theorems we may assume that $v\in H$, and by Lemma
\ref{lem:edgenontrivcent} we have that $\langle a^2\rangle$ centralises $N$. Hence $(N\times \langle
a^2\rangle) \backslash\langle z\rangle$ is a clique of size $|N|2^{k-2}-2$. Since $v$ acts on $N$ by
inversion and $|\C_G(v)|= 4$ we have for $|H|>8$ that $v\notin \langle a \rangle$ and so the kernel $M_1$ of
the action of $H$ on $N$ is either $\langle a^2\rangle$, $\langle a\rangle$ or $\langle a^2,av\rangle$. For
$|H|=8$ it is possible to have $v=a$ or $a^3$, in which case $\langle a^2,b\rangle$ and $\langle
a^2,ab\rangle$ are also possibilities for $M_1$.

Suppose first that $M_1=\langle a\rangle$. Then $N\times \langle a\rangle$ is an abelian group and so
$(N\times \langle a\rangle)\backslash \langle z\rangle$ is a clique of size $|G|/2-2$. For $g\in
G\backslash (N\times \langle a\rangle)$ we have that $g$ induces inversion on $N$ and conjugates $a$
to $a^{-1}$ or $a^{2^{k-2}-1}$. Thus $\C_G(g)\cong \langle g,z\rangle$ and so $\Gamma(G)$ contains
precisely $|G|/4$ isolated edges. Hence we have case \emph{\ref{phieq2}}(a).

Next suppose that $M_1=\langle a^2\rangle$. Since $v$ acts on $N$ by inversion, each of the $|H|/2$ elements
of $H\backslash \langle a^2,v\rangle$ induces an automorphism of $N$ of order 2 that is not inversion (see
(iii) of Lemma~\ref{lem:edgenontrivcent}). Hence it  centralises a nontrivial element of the abelian group $N$.
Therefore, as $N\ne1$  the elements from $[N\times\big(\langle a^2\rangle \cup (H\backslash \langle
a^2,v\rangle)\big)]\setminus\langle z\rangle$ form a connected component $\Delta$ of diameter at most three
on $\frac{|G|}{4}+\frac{|G|}{2}-2$ vertices. In addition,
  $|\N_G(\langle v,z\rangle)|=8$ implies there  are $|G|/8$ isolated edges
in $\Gamma(G)$ conjugate to $\{v,w\}$ and $\Gamma(G)=\frac{|G|}{8}K_2+\Delta$ where $\Delta$ is a
connected graph of diameter at most three. Thus we have case \emph{\ref{phieq22}}(b).

Next suppose that $M_1=\langle a^2,av\rangle$. Then, $(N\times M_1)\backslash \langle z\rangle$ has diameter
at most two.  Also, we may assume $v=b$. Moreover, the $|G|/4$ conjugates of $v$ provide $|G|/8$ isolated
edges. The elements of $Na^i$ for each  odd $i$ act on $N$ by inversion. For $|H|>8$ these elements commute
with $a^2\notin Z(G)$ and so $\Gamma(G)=\frac{|G|}{8}K_2+\Delta$ where $\Delta$ is a connected graph of
diameter two as in \emph{\ref{phieq22}}(b). When $|H|=8$ such elements provide another $|G|/8$ isolated
edges and so $\Gamma(G)=\frac{|G|}{4}K_2+K_{|G|/2-2}$ as in \emph{\ref{phieq2}}(a).

Finally suppose that $|H|=8$, $v=a$ or $a^3$, and $M_1=\langle a^2,b\rangle$ or $\langle a^2,ab\rangle$.
Then elements of $G\backslash (N\times M_1)$ act on $N$ by inversion and we obtain $|G|/4$  isolated edges.
Moreover, $(N\times M_1)\backslash \langle z\rangle$ is a clique since $M_1$ is abelian so we are in case
\emph{\ref{phieq2}}(a).
\end{proof}

\end{document}